\documentclass[bj]{imsart}

\usepackage{amssymb,amsmath,amsthm,mathdesign}
\RequirePackage[colorlinks,citecolor=blue,urlcolor=blue]{hyperref}

\arxiv{arXiv:1907.12096}

\newtheorem{theorem}{Theorem}[section]
\newtheorem{lemma}[theorem]{Lemma}
\newtheorem{proposition}[theorem]{Proposition}

\newtheorem{assumption}[theorem]{Assumption}
\newtheorem{remark}[theorem]{Remark}

\newtheorem{definition}[theorem]{Definition}

\newcommand{\be}{\begin{equation}}
\newcommand{\ee}{\end{equation}}

\newcommand{\bes}{\begin{equation*}}
\newcommand{\ees}{\end{equation*}}

\def\E{\bE}

\def\P{\bP} 

\def\cL{\mathcal{L}}

\def\bE{\mathbb{E}}

\newcommand{\R}{\mathbf{R}}

\renewcommand{\d}{{\rm d}}

\renewcommand{\geq}{\geqslant}
\renewcommand{\leq}{\leqslant}

\renewcommand{\P}{\mathrm{P}}

\def\m1{\mathbf{1}}

\begin{document}

\begin{frontmatter}

\title{The Osgood condition for stochastic partial differential equations.}

\runtitle{Osgood Condition}

\begin{aug}

\author{\fnms{Mohammud} \snm{Foondun}\thanksref{a,e1}\ead[label=e1,mark]{mohammud.foondun@strath.ac.uk}}
\author{\fnms{Eulalia} \snm{Nualart}\thanksref{b,e2}\ead[label=e2, mark]{eulalia.nualart@upf.edu}}

\address[a]{University of Strathclyde, Department of Mathematics and Statistics, 26 Richmond Street, Glasgow G1 1XH, Scotland.
\printead{e1}}

\address[b]{Universitat Pompeu Fabra and Barcelona Graduate School of Economics, Department of Economics and Business, Ram\'on Trias Fargas 25-27, 08005
Barcelona, Spain.
\printead{e2}}

\runauthor{Mohammud Foondun et al.}

\end{aug}

\begin{abstract}
We study the following equation 
\begin{equation*}
\frac{\partial u(t,\,x)}{\partial t}= \Delta u(t,\,x)+b(u(t,\,x))+\sigma \dot{W}(t,\,x),\quad t>0,
\end{equation*}
where $\sigma$ is a positive constant and $\dot{W}$ is a space-time white noise. The initial condition $u(0,x)=u_0(x)$ is assumed to be a nonnegative and continuous function. We first study the problem on $[0,\,1]$ with homogeneous Dirichlet boundary conditions.  Under some suitable conditions, together with a theorem of Bonder and Groisman in \cite{BonderG}, our first result shows that the solution blows up in finite time if and only if for some $a>0$,
\begin{equation*}
\int_{a}^\infty\frac{1}{b(s)}\,\d s<\infty,
\end{equation*}
which is the well-known Osgood condition.
We also consider the same equation on the whole line and show that the above condition is sufficient for the nonexistence of global solutions. Various other extensions are provided; we look at equations with fractional Laplacian and spatial colored noise in $\R^d$.
\\
\\
{\it AMS 2000 subject classications:} Primary 60H15, Secondary 60H10.
\end{abstract}

 
 \received{\smonth{1} \syear{0000}}


\end{frontmatter}
\section{Introduction and main results}

Consider the following non-linear heat equation,
\begin{equation*}
\left|\begin{aligned}
 \frac{\partial u(t, \,x)}{\partial t}&=
 \Delta u(t, x)+u(t,\,x)^{1+\eta},\quad x\in \R^d, \; t>0,\\
 u(0, x)&=u_0(x),
\end{aligned}\right.
\end{equation*}
where $u_0(x)$ is a nonnegative, continuous, and  bounded function.  It is well known that when $0<\eta\leq2/d$, there is no nontrivial global solution no matter how  small the nontrivial  initial condition $u_0$ is, while for $\eta>2/d$, one can construct nontrivial global solutions when $u_0$ is small enough; see \cite{Fujita, Haya, KST} for more precise statements and proofs. The exponent $\eta_c=2/d$ is often called the Fujita exponent after the author of the very influential paper \cite{Fujita}. When the equation is considered on the interval $[0,\,1]$ with homogeneous Dirichlet boundary conditions (ie. $u(t,0)=u(t,1)=0$), a different picture emerges. In this case, for any $\eta>0$, one can always construct nontrivial global solutions by taking $u_0$ small enough. And when $u_0$ is large enough, there is no global solution for any $\eta>0$; see  Theorem 17.3 of \cite{Souplet} for a precise statement.

One can ask whether the above phenomena still occur when one perturbs the equation with a noise term. For example, consider the following stochastic heat equation
\begin{equation}\label{usual}
\left|\begin{aligned}
\frac{\partial u(t, x)}{\partial t}&=
 \Delta u(t, x)+u(t,\,x)^{1+\eta}+\dot{W}(t,x)\\
 u(0, x)&=u_0(x),
 \end{aligned}\right.
 \end{equation}
 where $\dot{W}$ is a space-time white noise and the initial condition $u_0(x)$ is as above.
 More precisely, one can ask the following two questions.
\begin{itemize}
\item Does there exists a Fujita exponent? Or equivalently, for which values of $\eta$ one can find a nonnegative initial function so that there exist global solutions?
\item For the same equation on $[0,\,1]$ with homogeneous Dirichlet boundary conditions, can one take $u_0$ small enough so that there exist global solutions no matter what $\eta>0$ is?
\end{itemize}

For stochastic differential equations, the answer to the analogous question is given by Feller's test for explosions; see \cite[Chapter 5]{KS}. It is quite surprising that much less is known for stochastic partial differential equations. To the best of our knowledge, there are only two papers which look at these types of questions; \cite{BonderG} and \cite{DKT}. In the first paper the authors consider the equation on $[0,\,1]$  and give a negative answer to the second question above. In fact they consider the following more general equation
\begin{equation}\label{dirichlet}
\left|\begin{aligned}
\frac{\partial u(t,\,x)}{\partial t}&= \Delta u(t,\,x)+b(u(t,\,x))+\sigma \dot{W}(t,\,x),\quad x\in [0,\,1],\; t>0,\\
u(0,\,x)&=u_0(x),
\end{aligned}\right.
\end{equation}
with homogeneous Dirichlet boundary conditions. Here $\sigma>0$, $b:\R \rightarrow \R$ is a locally Lipschitz function, and the initial condition $u_0(x)$ is taken to be nonnegative and continuous; we will assume this throughout the whole paper.  The stochastic forcing term $\dot{W}$ is a space-time white noise. The main result of \cite{BonderG} says that the solution to (\ref{dirichlet}) blows up in finite time whenever $b$ is nonegative, convex, and satisfies the following well-known {\it Osgood condition}: for some $a>0$ 
\begin{equation}\label{osgood}
\int_{a}^{\infty} \frac{1}{b(s)}\,\d s<\infty,
\end{equation}
where $1/0=\infty$.
In \cite{DKT}, the authors investigate whether the Osgood condition is optimal. In particular, their Theorem 1.4 shows that if $\vert b(x)\vert=O(\vert x\vert \log \vert x\vert)$ as $\vert x \vert \rightarrow \infty$, then there exists a global solution to equation (\ref{dirichlet}).

As far as we know, the first question above has not been addressed till now. We briefly summarise the main findings of this current paper. For equation  \eqref{dirichlet}, we will show that Osgood condition (\ref{osgood}) is also necessary. Together with the result in \cite{BonderG}, this result shows the optimality of the Osgood condition.  We will then consider equation \eqref{usual} and answer the first question. We will show that $\eta_c=\infty$ meaning that there is no global solution no matter how small the initial condition is. This shows that the Fujita phenomenon does not occur in this stochastic setting.  In fact, we will show that the Osgood condition (\ref{osgood}) is sufficient for the nonexistence of global solutions for equation \eqref{usual}.

Before giving the main results of the paper, we provide some precision on various assumptions and technicalities.  We will need the following condition.
\begin{assumption}\label{A}
The function $b: \R\rightarrow \R^+$ is nonegative, locally Lipschitz and nondecreasing on $(0,\,\infty)$ and the initial condition $u_0(x)$ is nonnegative and continuous.
\end{assumption}
As in \cite{DKT}, we look at  random field solutions.
\begin{definition}\label{random-field}
A  local random field solution to \eqref{dirichlet} is a jointly measurable and adapted space-time process $u=\{u(t,x)\}_{(t,x) \in \R^+ \times[0,1]}$ satisfying the following integral equation
\begin{equation}\label{mild-dirichlet}
\begin{split}
u(t,\,x)=\int_0^1p(t,\,x,\,y)u_0(y)\,\d y&+\int_0^t\int_0^1p(t-s,\,x,\,y)b(u(s,\,y))\,\d y\,\d s\\\nonumber
& +\sigma\int_0^t\int_0^1p(t-s,\,x,\,y)W(\d y\,\d s),
\end{split}
\end{equation}
for all $t\in(0,\,\tau)$, where $\tau$ is some stopping time. If we can take $\tau=\infty$, then the local solution is also a global one. The function $p(t,x,y)$ is the heat kernel associated with the operator $\Delta$ with Dirichlet boundary conditions.  
\end{definition} 

As $u_0$ is continuous and $b$ is locally Lipschitz, the existence and uniqueness of  a local solution to equation (\ref{dirichlet}) is not an issue. Indeed, for each $N\geq 1$, one can define the truncation function 
\begin{equation} \label{bn}
b_N(x):={\bf 1}_{\{\vert x \vert \leq N\}} b(x)+ {\bf 1}_{\{\vert x \vert > N\}} b(N)+
{\bf 1}_{\{\vert x \vert < -N\}} b(-N)
\end{equation}
and obtain a unique global solution $\{u_N(t,x)\}_{(t,x) \in \R^+\times [0,1]}$ to equation (\ref{dirichlet}) where $b$ is replaced by $b_N$. Moreover, $u_N(t,x)$ is almost surely continuous in $(t,x)$. We consider the stopping time
\begin{equation*}
\tau_N:=\inf \left\{ t>0: \sup_{x\in [0,\,1]}|u_N(t,\,x)|>N\right\},
\end{equation*} 
where $\inf \emptyset:=\infty$. Then by the local property of the stochastic integral, one can easily show (see \cite[Section 4]{DKT}) that for each $N\geq \Vert u_0 \Vert_{\infty}$, we have a unique local random field solution $u(t,x)=u_N(t,x)$ for all $x \in [0,1]$ and $t\in [0,\,\tau_N)$. In particular, $u(t,x)$ is almost surely continuous in $(t,x)$. Moreover,  $\tau_N\leq \tau_{N+1}$.  Denote  $\tau_\infty=\lim_{N\rightarrow \infty}\tau_N$. If $\P(\tau_{\infty}<\infty)>0$, then we say that the solution blows up in finite time with positive probability and if  $\P(\tau_{\infty}<\infty)=1$, we say that the solution blows up in finite time almost surely.   Alternatively, since the noise is additive, one could also use a local inversion theorem to obtain local existence of solutions.  We are now ready to state the first result of this paper.
\begin{theorem}\label{theo-dirichlet}
Suppose that Assumption \ref{A} holds. If the solution to \eqref{dirichlet} blows up in finite time with positive probability then $b$ satisfies the Osgood condition \eqref{osgood}.
\end{theorem}
Together with the result of Bonder and Groisman, this can be seen as an extension of a similar result for stochastic differential equations with additive noise. Indeed in later case, Feller's test for explosions says that the Osgood condition is necessary and sufficient for blow up of the solution when the noise is a Brownian motion.  We will later describe a new method for proving this without appealing to Feller's test that works for a larger class of processes including the bifractional Brownian motion. This is due to \cite{leonvilla} which was also the inspiration for the proof of the above theorem.  
\vskip 11pt
We also consider equation (\ref{dirichlet}) in the whole line, that is, 
\begin{equation}\label{line}
\left|\begin{aligned}
\frac{\partial u(t,\,x)}{\partial t}&= \Delta u(t,\,x)+b(u(t,\,x))+\sigma \dot{W}(t,\,x)\quad x\in \R,\; t>0,\\
u(0,\,x)&=u_0(x).
\end{aligned}\right.
\end{equation}
As before, we look at the random field solution $u=\{u(t,x)\}_{(t,x) \in \R^+\times \R}$ which in this case, satisfies the following integral equation
\begin{equation}\label{mild-line}
\begin{split}
u(t,\,x)=\int_\R G(t,\,x,\,y)u_0(y)\,\d y&+\int_0^t\int_\R G(t-s,\,x,\,y)b(u(s,\,y))\,\d y\,\d s\\
&+\sigma\int_0^t\int_\R G(t-s,\,x,\,y)W(\d y\,\d s),
\end{split}
\end{equation}
where now $G(t,\,x,\,y)$ is the heat kernel associated with the Laplacian defined on the whole line. Here, the existence of a local solution is not straightforward. We will be more precise about this later.  Next, we describe our second main result which is a non-existence result. 

\begin{theorem}\label{line-blowup}
Suppose that Assumption \ref{A} holds. Then, if $b$ satisfies the Osgood condition \textnormal{(\ref{osgood})}, then almost surely, there is no global solution to equation \eqref{line}.
\end{theorem}
Here we use a completely different approach to that of \cite{BonderG}.  We use the almost sure growth properties of the stochastic term together with an observation borrowed from \cite{leonvilla} to arrive at our result. This observation is contained in the statement and proof of Proposition \ref{leon-villa} below. A key step in our strategy is to use the fact that for each $x$, the stochastic term in \eqref{mild-line} is a bifractional Brownian motion. We will use various continuity estimates as a well as the law of iterated logarithm for bifractional Brownian motion to arrive at the growth properties we need.

Our method is flexible enough so that some of the results above can extended to a wider class of equations. We describe these results next.  Consider the following equation 
\begin{equation}\label{dirichlet-general}
\left|\begin{aligned}
\frac{\partial u(t, x)}{\partial t}&=\cL u(t,x)+b(u(t,x))+ \dot{F}(t,\,x), \quad x \in B_1(0), \; t>0,\\
u(0,\,x)&=u_0(x),
\end{aligned}\right.
\end{equation}
where $B_r(z)$ denotes the (open) ball of center $z$ and radius $r$ in $\R^d$. Here $\cL$ is the generator of an $\alpha$-stable process killed upon exiting the ball $B_1(0)$ and $\dot{F}$ is a Gaussian noise which is white in time and has a spatial correlation given by the Riesz kernel. That is,
\begin{equation*} 
\E ( \dot{F}(t,x) \dot{F}(s,y))=\delta_0(t-s) f(x-y),
\end{equation*}
where $f(x)=\vert x \vert^{-\beta}$, $0<\beta<d$. 
The homogeneous Dirichlet boundary condition is given by
\begin{equation*}
u(t,x)=0 \qquad x \in \R^d \setminus B_1(0), \; \;t>0.
\end{equation*}

As before, the solution to equation (\ref{dirichlet-general}) is a jointly measurable adapted  random field $u=\left\{ u(t,x)\right\}_{t> 0, x \in  B_1(0)}$  satisfying the integral equation
\begin{equation}\label{mild:coloureda}
 \begin{split}
u(t,\,x)=\int_{B_1(0)} p_{\alpha}(t,\,x,\,y)u_0(y)\,\d y&+\int_0^t\int_{B_1(0)} p_{\alpha}(t-s,\,x,\,y)b(u(s,\,y))\,\d y\,\d s\\
&+\sigma\int_0^t\int_{B_1(0)} p_{\alpha}(t-s,\,x,\,y)F(\d y\,\d s),
\end{split}
\end{equation}
where $p_{\alpha}(t,\,x,\,y)$ is the Dirichlet fractional heat kernel. Recall that we have the following spectral decomposition
\begin{equation} \label{phi1}
p_{\alpha}(t,\,x,\,y)=\sum_{n=1}^{\infty} e^{-\lambda_n t} \phi_n(x) \phi_n(y)\quad\text{for all}\quad x,\,y \in B_1(0),\quad t>0,
\end{equation}
where $\{\phi_n\}_{n \geq 1}$ is an orthonormal basis of $L^2(B_1(0))$ and $0<\lambda_1<\lambda_2\leq \lambda_3 \leq \cdots$ is a sequence of positive numbers such that, for every $n \geq 1$,
\begin{equation*}
\begin{cases}
-(-\Delta)^{\alpha/2} \phi_n(x)=-\lambda_n \phi_n(x) \quad & x \in B_1(0),\\
\phi_n(x)=0 \quad &x \in \R^d\setminus B_1(0).
\end{cases}
\end{equation*}

As for equation (\ref{dirichlet}), one can easily show that if $\beta<\alpha$, then for each $N \geq \Vert u_0 \Vert_{\infty}$, there exists a unique local random field solution $u(t,x)$ to equation (\ref{dirichlet-general}) defined for all $x \in B_1(0)$ and $t \in [0, \tau_N)$, where
\begin{equation*}
\tau_N:=\inf \left\{ t>0: \sup_{x\in B_1(0)}|u_N(t,\,x)|>N\right\}, 
\end{equation*}
$\inf \emptyset:=\infty$,
and $u_N(t,x)$ is the solution to equation (\ref{dirichlet-general}) where $b$ is replaced by $b_N$ defined in (\ref{bn}).  The condition $\beta<\alpha$ ensures that the stochastic integral in (\ref{mild:coloureda}) is well-defined and almost surely continuous; see Remark \ref{r1} below.

The next result is the extension to equation  (\ref{dirichlet-general}) of Bonder and Groisman theorem and of Theorem \ref{theo-dirichlet}.
\begin{theorem}\label{Bonder-general}
Suppose that Assumption \ref{A} holds. Then if $b$ satisfies the Osgood condition \eqref{osgood} and under the additional assumption that $b$ is convex, the solution to \eqref{dirichlet-general} blows up in finite time almost surely. On the other hand, if the solution blows up in finite time with positive probability, then $b$ satisfies  the Osgood condition \eqref{osgood}.
\end{theorem}

\begin{remark}\label{r0}
The proof of the first part of Theorem \ref{Bonder-general} is an adaptation of the proof in \cite{BonderG}. But our method can do better, it can be used to prove that $\inf_{x\in B(0,\,1-\epsilon)}u(t,\,x)$ blows up in finite time for any $\epsilon>0$. We leave the proof for future work.
\end{remark}

\begin{remark} \label{r1}
Theorem \ref{Bonder-general} holds for a general spatial correlation $f$, where $f:\R^d \rightarrow \R$ is a nonnegative and nonnegative definite (generalized) function, continuous on $\R^d \setminus\{0\}$, integrable in a neighborhood of $0$, and whose Fourier transform $\mathcal{F}f=\mu$ is a tempered measure satisfying 
\begin{equation} \label{g}
\int_{\R^d} \frac{\mu(\d\xi)}{(1+\vert \xi \vert^{\alpha})^{\rho}} < \infty,
\end{equation}
for some $\rho \in (0,1)$,
where $(\mathcal{F} f)(\xi)=\int_{\R^d} f(y) e^{i \langle y, \xi \rangle} \d y$.
 Condition \eqref{g} with $\rho=1$ implies the existence and uniqueness of solutions; see Dalang \cite{Dalang}. The slightly more stringent condition \eqref{g} ensures that the solution is almost surely continuous as well; see Sanz-Sol\'e and Sarr\`a \cite{Sanz-Sarra}. 
In particular, when $f$ is the Riesz kernel, then $\mu(\d\xi)=c \vert \xi \vert^{-(d-\beta)}\d\xi$ and condition \eqref{g} holds for any $\rho>\beta/\alpha$ whenever $\beta < \alpha$.
\end{remark}

Consider now equation (\ref{dirichlet-general}) in the whole space, that is,
\begin{equation}\label{line-general}
\left|\begin{aligned}
\frac{\partial u(t, x)}{\partial t}&=\cL u(t,x)+b(u(t,x))+ \dot{F}(t,\,x), \quad x \in \R^d, \; t>0,\\
u(0,\,x)&=u_0(x),
\end{aligned}\right.
\end{equation}
where all the parameters are the same as above except that now $\cL$ is associated with an $\alpha$-stable process defined on the whole space.

As before, we are looking at random field solutions satisfying the following integral equation
\begin{equation}\label{mild-colored}
\begin{split}
u(t,\,x)=\int_{\R^d} G_{\alpha}(t,\,x,\,y)u_0(y)\,\d y&+\int_0^t\int_{\R^d} G_{\alpha}(t-s,\,x,\,y)b(u(s,\,y))\,\d y\,\d s\\
&+\sigma\int_0^t\int_{\R^d} G_{\alpha}(t-s,\,x,\,y)F(\d y\,\d s),
\end{split}
\end{equation}
where now $G_{\alpha}(t,\,x,\,y)$ is the heat kernel in $\R^d$ for the $\alpha$-stable process.  Our final theorem is as follows.
\begin{theorem}\label{line-genral-blowup}
Suppose that Assumption \ref{A} holds. Then, if $b$ satisfies the Osgood condition \eqref{osgood}, then almost surely, there is no global solution to equation \eqref{line-general}.
\end{theorem}
We end this introduction with some remarks concerning local existence of solutions when the equations are defined on the whole space. D. Khoshnevisan pointed to us that since for any fixed $t>0$, the last term of \eqref{mild-line} grows like $\sqrt{\log x}$ as $x$ goes to infinity, the solution to \eqref{mild-line} might blow up instantaneously. That is, any solution of \eqref{mild-line} can blow up for any $t>0$ so that there is no local solution. In the deterministic setting, similar phenomenon arises; see for instance \cite{Vasquez} where the exponential reaction-diffusion is studied. Proving such non-existence results is beyond the scope of this paper where the main concern is non-existence of global solution. The above result for instance makes no claim about the existence of a local solution. 

The rest of the paper is organized as follows. In Section 2 we give some preliminary results needed for the proofs of our results. Section 3 is devoted to the proofs of Theorems \ref{theo-dirichlet} and \ref{line-blowup}.  Theorems \ref{Bonder-general} and \ref{line-genral-blowup} are proved in Section 4. Finally, in Section 5 we discuss the extension of the results to  the  multiplicative noise case. 

\section{Preliminary information and estimates}
In this section, we give some background information needed for the proof of our results. We start off with a deterministic result about integral equations. This is taken from \cite{leonvilla} where it is used to show blow up for stochastic differential equations. We include a proof since it contains the main ideas of our method. 

\subsection{The Osgood condition for integral equations}

We start off with the following remark. Suppose that $b$ satisfies Assumption \ref{A} and consider the following integral equation  for $a \geq 0$
\begin{equation*}
y(t)=a+\int_0^tb(y(s))\,\d s, \quad t \geq 0.
\end{equation*}
By Picard-Lindel\"of theorem this equation admits a unique solution up to its blow up time defined as
\begin{equation*}
T:=\sup \{t>0: |y(t)|<\infty \},
\end{equation*}
where $\sup \emptyset :=-\infty$. Then we say that the solution blows up in finite time if $T<\infty$.
One can show  that this blow up time is equal to the following 
\begin{equation*}
\int_a^\infty\frac{1}{b(s)}\,\d s.
\end{equation*}
Therefore, we have that the solution blows up in finite time if and only of $\int_a^\infty\frac{1}{b(s)}\,\d s<\infty$.

We next consider the following assumption. In the upcoming sections, we will show that a large class of stochastic processes verify a similar condition.

\begin{assumption}\label{B}
$g: [0,\,\infty) \rightarrow \R$ is a continuous function such that 
\begin{equation*}
\limsup_{t\rightarrow \infty} \inf_{0\leq h\leq 1} g(t+h)=\infty.
\end{equation*}
\end{assumption}

\begin{proposition}\label{leon-villa}
Let $a\geq 0$ and suppose that Assumptions \ref{A} and \ref{B} hold.  Then the solution to the integral equation
\begin{equation}\label{ODE}
X_t=a+\int_0^tb(X_s)\,\d s+g(t)
\end{equation}
blows up in finite time if and only if the function $b$ satisfies the Osgood condition \eqref{osgood}.
\end{proposition}
 \begin{proof}
 Suppose that the solution blows up at finite time $T$. Since $g$ is continuous, we can set
 \begin{equation*}
M:=\sup_{0\leq s\leq T}|g(s)|.
 \end{equation*}
Let $0\leq t\leq T$. Upon noting that $b$ is nonnegative,  \eqref{ODE} gives
\begin{equation*}
X_t\leq a+M+\int_0^tb(X_s)\,\d s.
\end{equation*}
The nonnegativity of $b$ together with the continuity of $g$ imply that $X_t$ can only blow up to {\it positive infinity}. Let $Y_t=a+M+1+\int_0^tb(Y_s)\,\d s$. Then by a standard comparison result, we have $X_t\leq Y_t$ on $[0,\,T]$.  But since $X_t$ blows up at time $T$, $Y_t$ should also blow up by time $T$. This means that $b$ satisfies the Osgood condition \eqref{osgood}. 
 
 We now suppose that $X_t$ does not blow up in finite time. Let $\{ t_n\}_{n=1}^\infty$ be some sequence which tends to infinity. The nonnegativity of $b$ implies that
 \begin{align*}
 X_{t+t_n}&\geq a+\int_{t_n}^{t+t_n}b(X_s)\,\d s+g(t+t_n)\\
 &\geq a+\int_{0}^{t}b(X_{s+t_n})\,\d s+g(t+t_n)\\
 &\geq a+\inf_{0\leq h\leq 1}g(h+t_n)+\int_{0}^{t}b(X_{s+t_n})\,\d s,
\end{align*}
where the last inequality holds whenever $0\leq t \leq 1$. This means that $X_{t+t_n}\geq Z_t$ where
\begin{equation*}
Z_t=\frac{1}{2}\left(a+\inf_{0\leq h\leq 1}g(h+t_n)\right)+\int_{0}^{t}b(Z_s)\,\d s.
\end{equation*}
Since we are assuming that $X_t$ does not blow up in finite time, the blow up time of $Z_t$ has to be greater than 1, which implies that 
\begin{equation*}
\int_{\frac{1}{2}(a+\inf_{0\leq h\leq 1}g(h+t_n))}^\infty \frac{1}{b(s)}\,\d s>1.
\end{equation*}
But from Assumption \ref{B}, we can find a sequence $t_n\rightarrow \infty$ such that $\frac{1}{2}(a+\inf_{0\leq h\leq 1}g(h+t_n))\rightarrow \infty$. This  contradicts  the  Osgood condition (\ref{osgood})
and the proof is complete.
 \end{proof}
As mentioned in the introduction, the above result provides an alternative way to prove blow-up for stochastic differential equations of the following type,
\begin{equation*}
\d X_t=b(X_t)\,\d t+ \d B_t, \quad X_0=a,
\end{equation*}
where $B_t$ is a Brownian motion.
This can be written as the following integral equation,
\begin{equation*}
X_t=a+\int_0^tb(X_s)\,\d s+ B_t.
\end{equation*}
We can now show that almost surely $B_t$ satisfies Assumption \ref{B} above. Hence the Osgood condition is a necessary and sufficient condition for blow-up of the solution to the above equation. As showed in \cite{leonvilla}, one can replace the Brownian motion by a more general class of processes including the bifractional Brownian motion for which Feller's test for explosions is not applicable.

 \subsection{The bifractional Brownian motion and related results}
 
 The bifractional Brownian motion introduced in \cite{HoudreVilla} is a generalization of the
fractional Brownian motion. It is defined 
as a centered Gaussian process $B^{H,K}=(B_t^{H,K}, t \geq 0)$
with covariance
$$
R^{H,K}(t,s)=2^{-K} 
\left( (t^{2H}+s^{2H})^K- \vert t-s\vert^{2HK} \right),
$$
where $H \in (0,1)$ and $K \in (0,1]$. Note that if $K=1$, then $B^{H,1}$ is a fractional Brownian motion with Hurst parameter $H$.  The bifractional Brownian motion is H\"older continuous for any exponent less that $HK$. Moreover, it satisfies the following law of iterated logarithm;  see for instance Lemma 4.1 of  \cite{leonvilla} for an idea of the proof and further references. Set
$$
\psi_{H,K}(t):=t^{HK}\sqrt{2 \log \log t}, \quad t>e.
$$
\begin{lemma} \label{lil}
Almost surely,
$$
\limsup_{t \rightarrow \infty} \frac{B^{H,K}_t}{\psi_{H,K}(t)}=1.
$$
\end{lemma}

Consider now the process
 $$
 g(t,x):=\int_0^t\int_\R G(t-s,\,x,\,y)W(\d y\,\d s),
 $$
 where we recall that $G(t,\,x,\,y)$ denotes the Gaussian heat kernel.
Clearly the above is the solution to the stochastic heat equation \eqref{line} with zero drift, zero initial condition, and $\sigma=1$. 

It is shown in \cite{LeiNualart} that for a fixed $x\in \R$, the process $(g(t,x), t\geq 0)$ is a bifractional Brownian motion with parameters $H=K=\frac12$ multiplied by a constant.   In fact, the covariance of $g(t,x)$ is given by
 $$
 \E (g(t,x) g(s,x))=\frac{1}{\sqrt{2\pi}}(\sqrt{t+s}-\sqrt{\vert t-s\vert}).
 $$
In particular, the process $(g(t,x), t \geq 0)$ is H\"older continuous for any exponent less than $HK=1/4$. 

 The following estimates on the increments of $g(t,\,x)$ are well known. For instance see Theorem 6.7 in p.28 of \cite{minicourse} and the proof of  Corollary 3.4 in \cite{Walsh}.
 \begin{lemma} \label{holder}
 For all $p \geq 2$ there exist constants $c_p, \tilde{c}_p>0$ such that for all $x,y \in \R$ and $s,t \geq 0$,
 \begin{equation*}
 \sup_{t>0} \E\left[\vert g(t,x)-g(t,y) \vert^p\right] \leq c_p \vert x-y \vert^{p/2}
 \end{equation*} 
 and
 \begin{equation*}
 \sup_{x \in \R} \E\left[\vert g(t,x)-g(s,x) \vert^p\right] \leq \tilde{c}_p \vert s-t\vert^{p/4}.
 \end{equation*}
 \end{lemma}

 As a consequence of Lemma \ref{holder} and the improvement of the classical Garsia's lemma obtained in Proposition A.1. of \cite{DKN}, we have the following estimate.
 \begin{proposition} \label{supn}
 For all $p \geq 2$, there exists a constant  $A_p>0$ such that 
 for any integer $n \geq 1$,
 $$
  \E\left[\sup_{s,t \in [n,n+2], x,y \in [0,1]}\vert g(t,x)-g(s,y) \vert^p \right] \leq A_p 2^{p/4}.
  $$ 
 \end{proposition}
 
 \begin{proof}
 The proof follows along the same lines as the proof of Lemma 4.5 in \cite{DKN}.
 Indeed, by Lemma \ref{holder} and Proposition A.1. in \cite{DKN}, we get that for all $p \geq 2$,  there exists a constant  $A_p>0$ such that for any $\epsilon>0$,
 $$
  \E\left[\sup_{(\vert t-s\vert^{1/2}+\vert x-y \vert)^{1/2} \leq \epsilon}\vert g(t,x)-g(s,y) \vert^p \right] \leq A_p \epsilon^p.
 $$
 Then, using this inequality with $\epsilon=\sqrt{2} 2^{1/4}$ implies the desired result.
 \end{proof}
 
We can now use Proposition \ref{supn} to get the following almost sure result. This is an extension to the multiparameter case of  Lemma 4.2 in \cite{leonvilla}.
\begin{proposition} \label{supnas}
Almost surely,
 $$
   \sup_{s,t \in [n,n+2], x,y \in [0,1]}
  \frac{\vert g(t,x)-g(s,y) \vert}{\psi_{\frac12, \frac12}(n)} \longrightarrow 0, \quad \text{ as } n \rightarrow \infty.
  $$ 
 \end{proposition}
 \begin{proof}
 Proposition \ref{supn} implies that for $p>4$,
 $$
 \E \left[\sum_{n=1}^{\infty} \sup_{s,t \in [n,n+2], x,y \in [0,1]}
  \frac{\vert g(t,x)-g(s,y) \vert^p}{\psi_{\frac12, \frac12}(n)^p}\right] 
  \leq \sum_{n=1}^{\infty} \frac{A_p 2^{p/4}}{\psi_{\frac12, \frac12}(n)^p} < \infty,
 $$
 which gives us the desired result.
 \end{proof}
 
 As a consequence of Proposition \ref{supnas}, we get the following estimate.
 \begin{proposition}\label{sup2}
 Almost surely, there exists a sequence $t_n \rightarrow \infty$ such that 
 $$
 \inf_{h \in [0,1], x \in [0,1]} g(t_n+h,x) \rightarrow \infty \quad\text{as}\quad n\rightarrow \infty.
 $$
 \end{proposition}
 \begin{proof}
 Fix $x_0 \in [0,1]$. Choose $\omega$ such that both Proposition \ref{supnas} and Lemma \ref{lil} hold. We now write
 \begin{align*}
 \inf_{h \in [0,1], x \in [0,1]} g(t+h,x)
 &=g(t,x_0)+\inf_{h \in [0,1], x \in [0,1]} \left(g(t+h,x)-g(t,x_0)\right) \\
 &\geq g(t,x_0)+\inf_{h \in [0,1], x \in [0,1]} \left(-\vert g(t+h,x)-g(t,x_0)\vert\right) \\
  &\geq  \frac{g(t,x_0)}{\psi_{\frac12, \frac12}(t)} \psi_{\frac12, \frac12}(t)-\sup_{h \in [0,1], x \in [0,1]} \frac{\vert g(t+h,x)-g(t,x_0)\vert}{\psi_{\frac12, \frac12}([t])} \psi_{\frac12, \frac12}([t]).
 \end{align*} 
We use Proposition \ref{supnas} and Lemma \ref{lil} to choose an appropriate sequence $t_n$ and finish the proof.
 \end{proof}
 
\section{Proofs of Theorems \ref{theo-dirichlet} and \ref{line-blowup}}

As mention earlier, the proof of Theorem \ref{theo-dirichlet} follows that of Proposition \ref{leon-villa} but it heavily relies on the fact that the stochastic term in the random field formulation is continuous and that the equation itself is defined on an interval.
\begin{proof}[Proof of Theorem \ref{theo-dirichlet}]
Set 
\begin{equation*}
T:=\sup\{ t>0: \sup_{x\in [0,\,1]}|u(t,\,x)|<\infty\},
\end{equation*}
where $\sup \emptyset:=-\infty$.
Since the solution blows up in finite time with positive probability, we can find a set $\Omega$  satisfying $\P(\Omega)>0$ such that for any $\omega\in \Omega$, we have $T(\omega)<\infty.$  We now fix such an $\omega$ but for the sake of notational convenience, we won't indicate the dependence on $\omega$ in what follows. We recall that we are looking at the mild formulation
\begin{equation*} \begin{split}
u(t,\,x)=\int_0^1p(t,\,x,\,y)u_0(y)\,\d y&+\int_0^t\int_0^1p(t-s,\,x,\,y)b(u(s,\,y))\,\d y\,\d s\\
&+\sigma\int_0^t\int_0^1p(t-s,\,x,\,y)W(\d y\,\d s).
\end{split}
\end{equation*}
The third term in the above display is almost surely continuous. Therefore the following quantity below is finite almost surely
\begin{equation*}
M:=\sup_{x\in [0,\,1]\, t\in (0, T]}\left|\int_0^t\int_0^1p(t-s,\,x,\,y)W(\d y\,\d s)\right|.
\end{equation*}
Moreover by the nonnegativity of $b$ and the initial condition, we have
\begin{equation*}
u(t,\,x)\geq \sigma\int_0^t\int_0^1p(t-s,\,x,\,y)W(\d y\,\d s).
\end{equation*}
This means that  
\begin{equation*}
\inf_{t\in [0,\,T], x\in [0,\,1]}u(t,\,x)\geq -\sigma M.
\end{equation*}
Since $u_0$ is bounded, we have
\begin{equation*}
\left|\int_0^1p(t,\,x,\,y)u_0(y)\,\d y \right|\leq a,
\end{equation*}
for some positive constant $a$. Denote $\mathcal{A}:=\{s\in (0,\,t), y\in (0,\,1); -\sigma M\leq u(s, y)\leq 0\}$ and $\mathcal{B}:=\{s\in (0,\,t), y\in (0,\,1); u(s, y)>0 \}$ and write 
\begin{align*}
 \int_0^t\int_0^1p(t-s,\,x,\,y)b(u(s,\,y))\,\d y\,\d s&=\iint_\mathcal{A}p(t-s,\,x,\,y)b(u(s,\,y))\,\d y\,\d s\\
&+\iint_\mathcal{B}p(t-s,\,x,\,y)b(u(s,\,y))\,\d y\,\d s\\
&:=I_1+I_2.
\end{align*}
Since we are assuming that $b$ is nonnegative and nondecreasing on $(0, \infty)$, this immediately gives us
\begin{equation*}
I_2\leq \int_0^tb(Y_s)\,\d s,
\end{equation*}
where $Y_t:=\sup_{x\in [0,\,1]}u(t,\,x).$
Since $b$ is assumed to be continuous, we have $I_1\leq K$, where $K$ is an almost sure finite quantity.  Putting all these estimates together, we obtain 
\begin{align*}
Y_t\leq a+\sigma M+K+\int_0^tb(Y_s)\,\d s.
\end{align*}
We can now proceed as in the proof of Proposition \ref{leon-villa} to conclude the proof.
\end{proof}
\begin{remark}
We remark that the blow up time $T$ used above is the same as $\tau_\infty$ defined in the introduction. Indeed if $T<\tau_{\infty}$, then $T<t<\tau_{\infty}$ for some $t$.
By the  definition of  $T$, we should have $\sup_{x \in [0,1]} \vert u(t,x) \vert=\infty$, but then this would imply (by the  definition of $\tau_N$) that $t>\tau_N$ for any $N$. Thus, $t> \tau_{\infty}$, which contradicts the  assumption. Therefore, we have that  $\tau_{\infty} \leq T$. If $\tau_{\infty}<T$, then $\tau_{\infty}<t<T$ for some $t$.  As $t>\tau_{\infty}$, we get $\sup_{x \in [0,1]} \vert u(t,x) \vert=\infty$ which contradicts the fact  that  $t<T$. Therefore, we conclude that $T=\tau_{\infty}$.
\end{remark}

\begin{proof}[Proof of Theorem \ref{line-blowup}]
Let $\{t_n\}$ be a sequence of positive numbers which we are going to choose later. From the mild formulation of the solution and the nonnegativity of the function $b$, we obtain 
\begin{equation*}\begin{split}
u(t+t_n,\,x)&=\int_\R G(t+t_n,\,x,\,y)u_0(y)\,\d y+\int_0^{t+t_n}\int_\R G(t+t_n-s,\,x,\,y)b(u(s,\,y))\,\d y\,\d s\\ 
&\qquad  +\sigma\int_0^{t+t_n}\int_\R G(t+t_n-s,\,x,\,y)W(\d y\,\d s)\\
&\geq \int_\R G(t+t_n,\,x,\,y)u_0(y)\,\d y+\int_0^{t}\int_\R G(t-s,\,x,\,y)b(u(s+t_n,\,y))\,\d y\,\d s\\ 
&\qquad +\sigma\int_0^{t+t_n}\int_\R G(t+t_n-s,\,x,\,y)W(\d y\,\d s).
\end{split}
\end{equation*}
We will take $0\leq t\leq 1$ and $x \in (0,1)$. Recall that 
\begin{equation*}
g(t+t_n,\,x):=\int_0^{t+t_n}\int_\R G(t+t_n-s,\,x,\,y)W(\d y\,\d s).
\end{equation*}
Hence by Proposition \ref{sup2}, we can find a sequence $t_n \rightarrow \infty$ so that the above quantity is positive for $0\leq t\leq 1$ and $x \in (0,1)$. Therefore $u(t+t_n,\,x)$ is also positive for any $x\in (0,\,1)$ and any  $0\leq t\leq 1$. We now use the fact that $b$ is nondecreasing on $(0,\,\infty)$ to bound the second term as follows. For fixed $x\in (0, \,1)$,
\begin{align*}
\int_0^{t}\int_\R G(t-s,\,x,\,y)&b(u(s+t_n,\,y))\,\d y\,\d s\\
&\geq \int_0^t b\left(\inf_{y\in (0,\,1)}u(s+t_n,\,y)\right)\int_{(0,\,1)}G(t-s,\,x,\,y)\,\d y\,\d s\\
&\geq \int_0^t b\left(\inf_{y\in (0,\,1)}u(s+t_n,\,y)\right)\,\d s,
\end{align*}
where we have used that fact that $G(t,\,x,\,y)\geq \frac{c}{t^{1/2}}$ whenever $|x-y|\leq t^{1/2}$.  We now set $Y_t:=\inf_{y\in (0,\,1)}u(t+t_n,\,y)$ and combine the above estimates to obtain 
\begin{align*}
Y_t\geq \inf_{0\leq h\leq 1, x\in (0,\,1)}\Big\{ \int_\R G(h+t_n,\,x,\,y)u_0(y)\,\d y+\sigma g(h+t_n,\,x)\Big\}+\int_0^tb(Y_s)\,\d s.
\end{align*}
We now choose $\omega$ as in Proposition \ref{sup2}, and we can therefore find a sequence $t_n\rightarrow \infty $ such that $\inf_{0\leq h\leq 1, x\in (0,\,1)}g(h+t_n,\,x)$ goes to infinity. By the proof of Proposition \ref{leon-villa}, we have the required result.
\end{proof}
\section{Extension to fractional Laplacian and colored noise}
The aim of this section is to prove Theorems  \ref{Bonder-general} and \ref{line-genral-blowup}. For this, we first define rigorously the Gaussian noise $F$ and extend the results of Section 2.2 to the equation in $\R^d$.

\subsection{The Gaussian noise $F$}

Let $\mathcal{D}(\R_+\times \R^d)$ be the space of real-valued infinitely differentiable functions with compact support.
Following \cite{Dalang}, on a complete probability space $(\Omega, \mathcal{F}, \P)$, we consider a centered Gaussian family of random variables $\{F(\varphi), \varphi \in \mathcal{D}(\R_+ \times \R^d)\}$ with covariance
$$
\E\left[ F(\varphi) F(\psi)\right]=\int_{\R_+ \times \R^{2d}} \varphi(t,x) \psi(t,y) f(x-y) \d x \d y \d t,
$$
where $f$ is as in Remark \ref{r1}.
Let $\mathcal{H}$ be the completion of $\mathcal{D}(\R_+\times \R^d)$ with respect to the inner product
\begin{equation} \label{p}\begin{split}
\langle \varphi, \psi \rangle_{\mathcal{H}}&=\int_{\R_+ \times \R^{2d}} \varphi(t,x) \psi(t,y) f(x-y) \d x \d y \d t \\
&=\int_{\R_+ \times \R^{d}} \mathcal{F}\varphi(t,\cdot) (\xi)\overline{\mathcal{F}\psi(t,\cdot) (\xi)}\mu(\d\xi) \d t,
\end{split}
\end{equation}
where the last equality follows by Parseval's identity.
The mapping $\varphi \mapsto F(\varphi)$ defined in $\mathcal{D}(\R_+ \times \R^d)$ extends to a linear isometry 
between $\mathcal{H}$ and the Gaussian space spanned by $F$. We will denote the isometry by 
$$
F(\varphi)=\int_{\R_+ \times \R^d} \varphi(t,x) F(\d t\,\d x), \qquad \varphi \in \mathcal{H}.
$$
Notice that if $\varphi, \psi \in \mathcal{H}$, then $\E\left[ F(\varphi) F(\psi)\right]=\langle \varphi, \psi \rangle_{\mathcal{H}}$. Moreover,
$\mathcal{H}$ contains the space of measurable functions $\phi$ on $\R_+ \times \R^d$ such that
$$
\int_{\R_+ \times \R^{2d}} \vert \phi(t,x) \phi(t,y)\vert  f(x-y) \d x \d y \d t < \infty.
$$

\subsection{Estimates for the whole space}

Consider the solution to the stochastic heat equation \eqref{line-general} with zero drift, zero initial condition,
and $\sigma=1$, that is,
 $$
 g_{\alpha,\beta}(t,x):=\int_0^t\int_{\R^d} G_{\alpha}(t-s,\,x,\,y)F(\d y\,\d s),
 $$ 
 where recall that  $G_{\alpha}(t,\,x,\,y)$ is the fractional heat kernel in $\R^d$ and $F$ has a spatial correlation given by the Riesz kernel.
 
 Let us compute the covariance of the Gaussian process $(g_{\alpha, \beta}(t,x), t \geq 0)$ for $x \in \R^d$ fixed.
 By \eqref{p}, as 
 \begin{equation} \label{fourier}
 \mathcal{F} G_{\alpha}(t, \,x,\,\cdot)(\xi)=e^{i\langle x, \xi \rangle -\frac12 t \vert \xi \vert^{\alpha}},\quad \xi \in \R^d,
 \end{equation}
 we get that, for $s \leq t$,
 \begin{equation*} \begin{split}
 &\E (g_{\alpha, \beta}(t,x) g_{\alpha, \beta}(s,x))\\
 &=
\int_0^s  \int_{\R^d \times \R^d}  G_{\alpha}(t-u,\,x,\,y)
G_{\alpha}(s-u,\,x,\,z) \vert z-y \vert^{-\beta} \d y \d z \d u \\
&=c_{d,\beta} \int_0^s \int_{\R^d }  \vert \xi \vert^{-(d-\beta)}
e^{-\frac12 (t-u) \vert \xi \vert^{\alpha}}
e^{-\frac12 (s-u) \vert \xi \vert^{\alpha}} \d \xi \d u\\
&=c_{d,\beta}\int_{\R^d }  \vert \xi \vert^{-(d-\beta)-\alpha}
e^{-\frac12 (t+s) \vert \xi \vert^{\alpha}}
\left(e^{s \vert \xi \vert^{\alpha}}-1\right) \d\xi\\
&=c_{d,\beta}\int_{\R^d }  \vert \xi \vert^{-(d-\beta)-\alpha}
\left(e^{-\frac12 (t-s) \vert \xi \vert^{\alpha}}-e^{-\frac12 (t+s) \vert \xi \vert^{\alpha}}\right) \d\xi\\
&=c_{d,\beta}\int_{\R^d }  \vert \xi \vert^{-(d-\beta)-\alpha}
\left(\int_{-\frac12 (t+s) \vert \xi \vert^{\alpha}}^0 e^z \d z-\int_{-\frac12 (t-s) \vert \xi \vert^{\alpha}}^0 e^z \d z\right) \d\xi \\
&=c_{d,\beta,\alpha}
\left((t+s)^{1-\frac{\beta}{\alpha}}-(t-s)^{1-\frac{\beta}{\alpha}}\right),
 \end{split}
\end{equation*}
where $c_{d,\beta,\alpha}=c_{d,\beta} \int_{\R^d }  \vert \xi \vert^{-(d-\beta)-\alpha}(1-e^{-\frac{1}{2} \vert \xi \vert^{\alpha}} )\d\xi$.

Therefore, for $x \in \R^d$ fixed, the process $(g(t,x), t \geq 0)$ is a bifractional 
 Brownian motion with parameters $H=\frac{\alpha-\beta}{2}$ and $K=\frac{1}{\alpha}$, multiplied by a constant. 
 In particular, it is H\"older continuous for any exponent less than $HK=\frac{\alpha-\beta}{2 \alpha}$. 
 
The next proposition is the extension of Lemma \ref{holder} and Propositions \ref{supn}, \ref{supnas} and \ref{sup2} to the process $g_{\alpha,\beta}$.
 \begin{proposition} \label{holderbis}
 \begin{itemize}
 \item [\textnormal{(a)}] For all $p \geq 2$ there exists constants $c_p, \tilde{c}_p>0$ such that for all $x,y \in \R^d$ and $s,t \geq 0$,
 \begin{equation*} 
 \sup_{t>0} \E\left[\vert g_{\alpha,\beta}(t,x)-g_{\alpha,\beta}(t,y) \vert^p\right] \leq c_p \vert x-y \vert^{\frac{(\alpha-\beta)p}{2}}
 \end{equation*} 
 and
 \begin{equation*}
 \sup_{x \in \R^d} \E\left[\vert g_{\alpha,\beta}(t,x)-g_{\alpha,\beta}(s,x) \vert^p\right] \leq \tilde{c}_p \vert s-t\vert^{\frac{(\alpha-\beta)p}{2\alpha}}.
 \end{equation*}

\item [\textnormal{(b)}] For all $p \geq 2$, there exists a constant  $A_p>0$ such that 
 for any integer $n \geq 1$,
 $$
  \E\left[\sup_{s,t \in [n,n+2], x,y \in B_1(0)}\vert g_{\alpha,\beta}(t,x)-g_{\alpha,\beta}(s,y) \vert^p \right] \leq A_p 2^{\frac{(\alpha-\beta)p}{2 \alpha}}.
  $$ 
\item [\textnormal{(c)}] Almost surely,
 $$
   \sup_{s,t \in [n,n+2], x,y \in B_1(0)}
  \frac{\vert g_{\alpha,\beta}(t,x)-g_{\alpha,\beta}(s,y) \vert}{\psi_{\frac{\alpha-\beta}{2}, \frac{1}{\alpha}}(n)} \longrightarrow 0, \quad \text{ as } n \rightarrow \infty.
  $$ 
 
 \item [\textnormal{(d)}] Almost surely, there exists a sequence $t_n \rightarrow \infty$ such that 
 $$
 \inf_{h \in [0,1], x \in B_1(0)} g_{\alpha,\beta}(t_n+h,x) \rightarrow \infty.
 $$

\end{itemize} 
 \end{proposition}
 
 \begin{proof}
 The proof of (a) follows from \cite{Sanz-Sarra} and (\ref{fourier}). Moreover, (a) implies that Lemma 4.5 in \cite{DKN} also holds for our process $g_{\alpha,\beta}(t,x)$, which gives (b). Finally, (c) and (d) follow as in the proof of Propositions \ref{supnas} and \ref{sup2}.
 \end{proof}
 
 \subsection{Proof of Theorems \ref{Bonder-general} and \ref{line-genral-blowup}}
The proof of the first part of Theorem \ref{Bonder-general} follows as in the proof of the main result in \cite{BonderG}, but a different constant arises in Feller's test. The proof of the second part
follows along the  same lines as the proof of Theorem \ref{theo-dirichlet} and is therefore omitted. 
\begin{proof}[Proof of the first part of Theorem \ref{Bonder-general}]
As in  \cite{BonderG}, we set
$$
Y_t=\int_{B_1(0)} u(t,x) c \phi_1(x) \d x
$$
where $u(t,x)$ is the local solution to (\ref{mild:coloureda}), $\phi_1$ is defined in (\ref{phi1}) and $c^{-1}=\int_{B_1(0)} \phi_1(x) dx$. 
Recall that $\phi_1(x)>0$ for all $x \in B_1(0)$, see for e.g. \cite[Theorem 4.2]{CS97}.
Then we obtain that $Y_t \geq X_t$ a.s., where $X_t$ is the solution to  the stochastic differential equation
$$
dX_t=(-\lambda_1 X_t+b(X_t)) \d t+ \d Z_t, \qquad X_0=Y_0
$$
and
$$
Z_t:=\int_0^t \int_{B_1(0)} c \phi_1(y) F(\d y\,\d s).
$$ 
Finally, we can use Feller's test for explosion as in \cite{BonderG} to show that
$X_t$ explodes in finite time with probability one. In fact, it suffices to
consider as scale function
$$
p(x)=\int_0^x \exp \left( -\frac{2}{\kappa}\int_0^s (-\lambda_1 \xi+b(\xi))  d\xi \right)
ds,
$$
where $\kappa:=c^2\int_{B_1(0) \times B_1(0)} \phi_1(y )\phi_1(z)\vert y-z\vert^{-\beta} \d y \d z$,
as $Z_t=\sqrt{\kappa} B_t$, where  $B_t$ is a Brownian motion.
\end{proof}

The proof of Theorem \ref{line-genral-blowup} is similar to that of Theorem \ref{line-blowup}. We will indicate the differences only. 

\begin{proof}[Proof of Theorem \ref{line-genral-blowup}]
The proof follows that of Theorem \ref{line-blowup}. We use the last part of Proposition \ref{holderbis} together with the following inequality 
\begin{align}
G_\alpha(t,\,x,\,y)\geq \frac{\text{c}}{t^{d/\alpha}}\quad\text{whenever}\quad |x-y|\leq t^{1/\alpha}
\end{align}
to prove the required result. See for instance \cite{kolo} for justifications of the above inequality.  We leave it to the reader to fill in the details.
\end{proof}

\section{Extension to multiplicative noise}

In this section we discuss the extension of the previous results when the constant $\sigma$ is replaced
by a locally Lipschitz function $\sigma:\R\rightarrow \R$. For equations on bounded domains, we believe that the Osgood condition is necessary and sufficient for finite time blow-up under the condition $\frac{1}{K}\leq \sigma(x)\leq K$ for all $x \in \R$, for some constant $K>0$. We leave this for future work but the following extension is straightforward.  We start  with equations  (\ref{dirichlet}) and (\ref{dirichlet-general}) on $[0,1]$ and $B_1(0)$, respectively. In this case, the proof of Theorem \ref{theo-dirichlet} and the second half of Theorem \ref{Bonder-general}  extend easily if $\sigma$ is bounded below and above by positive constants. This yields the  following result.
  \begin{theorem} \label{last}
Consider equations  \eqref{dirichlet} or \eqref{dirichlet-general} with $\sigma$ replaced
by $\sigma(u(t,x))$, where $\sigma:\R\rightarrow \R$ is a locally Lipschitz function satisfying $\frac{1}{K}\leq \sigma(x)\leq K$ for all $x \in \R$, for some constant $K>0$. Suppose that Assumption \ref{A} holds. If the corresponding solution blows up in finite time with positive probability, then $b$ satisfies the Osgood condition \eqref{osgood}.
 \end{theorem}
 However, the proof of the  first half of Theorem \ref{Bonder-general} which follows  by  Bonder and Groisman's method does not directly extend to  the multiplicative noise case. We can apply the method used in the proof of Theorem \ref{line-blowup} but we require a new idea since the stochastic term is no longer Gaussian.
 
 Consider now equation (\ref{line}) on the real line and assume  that 
 $\sigma$ is replaced by a locally Lipschitz function $\sigma:\R\rightarrow \R$ bounded away from zero and infinity. Then, the statement of Theorem \ref{line-blowup} holds true provided that we prove Proposition \ref{sup2} for
$$
 g(t,x):=\int_0^t\int_\R G(t-s,\,x,\,y) \sigma(u(s,y)) W(\d y\,\d s).
 $$
 We believe that this is true but the proof is out of the scope of this paper and is left  for further work. The same discussion applies for the equation on the whole space (\ref{line-general}) and Theorem \ref{line-genral-blowup}.

\bibliography{Foon-Nual}

\end{document}